\documentclass[preprint,12pt]{elsarticle}
\usepackage{graphicx} 
\usepackage{amsmath, amssymb, amsthm, tikz}
\usepackage{nicematrix}
\usepackage{mathtools}
\usepackage{blkarray, bigstrut}
\usepackage{gauss}
\usepackage{xcolor}
\usepackage[margin = 2cm]{geometry}
\usepackage[colorlinks = true]{hyperref}

\newcommand{\rvline}{\hspace*{-\arraycolsep}\vline\hspace*{-\arraycolsep}}
\date{}
\newtheorem{lemma}{\bf Lemma}[section]
\newtheorem{example}{\bf Example}[section]

\newtheorem{theorem}{\bf Theorem}[section]
\newtheorem{proposition}[lemma]{\bf Proposition}

\newtheorem{definition}{\bf Definition}[section]
\newtheorem{remark}{\bf Remark}[section]

\journal{~}

\begin{document}
	
\begin{frontmatter}
		
		
 %

\author[add1]{Tapa Manna}
  \ead{mannatapa24@gmail.com}
\author[add2]{Sauvik Poddar}
  \ead{sauvikpoddar1997@gmail.com}
\author[add2]{Angsuman Das\corref{cor1}}
  \ead{angsuman.maths@presiuniv.ac.in}
\author[add1]{Baby Bhattacharya}
  \ead{babybhatt75@gmail.com}

  \cortext[cor1]{Correspondence Author}
  \address[add1]{Department of Mathematics,\\ NIT Agartala, India}
  \address[add2]{Department of Mathematics,\\ Presidency University, Kolkata, India}

\title{Prime Order Element Graph of a Group - II}

\begin{abstract}
In this sequel paper, we continue the analysis of the prime order element graph $\Gamma(G)$ of a finite group $G$, where vertices are elements of $G$ and edges connect distinct elements $x, y$ satisfying $\circ(xy) = p$ for some prime $p$. Our investigation focuses on the adjacency and Laplacian spectra, planarity, and clique number of this graph. We conclude by outlining open issues and potential directions for future investigations.
\end{abstract}
		
\end{frontmatter}

\section{Introduction}
The fundamental reason for associating graphs with groups is to facilitate the study of the group by translating its algebraic properties into the geometric/combinatorial language of graph theory. This translation process involves mapping group elements or subgroups to vertices and defining edges based on specific algebraic relations (e.g., commutativity, generation, difference sets). The literature features a variety of graphs defined on groups, such as Power Graphs 
\cite{Cameron-Ghosh}, Commuting graphs \cite{commuting-graph}, Gruenberg-Kegel graphs, Difference Graphs \cite{difference_graph}, and Comaximal Subgroup graphs \cite{das-saha}, all serving distinct purposes. (See \cite{Cameron-survey} for a comprehensive survey.) In this paper, we continue our study of one such graph called prime order element graph $\Gamma(G)$ of a finite group $G$ introduced in \cite{tapa1}. The forbidden subgraph characterization of this graph can be found in \cite{tapa2}. For terms and definitions related to graph theory, one can refer to \cite{west-graph-book}.

\begin{definition} Let $G$ be a finite group. The prime-order element graph $\Gamma(G)$ of a group $G$ is defined to be a graph with $G$ as the set of vertices and two distinct vertices $x$ and $y$ are adjacent if and only if order of $xy$ is prime.	
\end{definition}

In Section \ref{adj-spectrum} and \ref{lap-spectrum}, we investigate the adjacency and Laplacian spectrum of the graph $\Gamma(G)$. In Section \ref{planar-section}, we characterize the groups $G$ for which $\Gamma(G)$ is planar. In Section \ref{clique-section}, we discuss about the clique number of $\Gamma(G)$ when $G$ is an abelian $p$-group. Finally we conclude with some open issues in Section \ref{conclusion-section}.

\section{Adjacency Spectrum of $\Gamma(G)$}\label{adj-spectrum}
A graph $\Gamma$ is said to be \textit{integral}/\textit{Laplacian integral} if all the eigenvalues of the adjacency/Laplacian matrix of $\Gamma$ are integers. In this section, we prove that $\Gamma(G)$ is integral for an abelian $2$-group $G$, whereas for a cyclic $p$-group $G$ with $p$ an odd prime, $\Gamma(G)$ has an irrational eigenvalue.
\begin{theorem}\label{abelian-2-group-integral}
If $G$ is an abelian $2$-group, then the prime order element graph $\Gamma(G)$ is integral.
\end{theorem}
\begin{proof}
From Theorem 2.7 in \cite{tapa1}, $\Gamma(G)$ is connected if and only if $G\cong \mathbb{Z}^n_2$. In this case, $\Gamma(G)$ is complete and have eigenvalues $2^n-1$ with multiplicity $1$ and $-1$ with multiplicity $2^n-1$. Then we consider the case when $\Gamma(G)$ is disconnected, i.e., $G\not\cong \mathbb{Z}^n_2$. It is enough to show that eigenvalues of each component of $\Gamma(G)$ are integers.
 
Let $S$ be the subset of $G$ consisting of elements of order $2$. As an element of order $2$ is not adjacent to any element of order $>2$, $A=S\cup \{e\}$ is a component of $\Gamma(G)$. Clearly any two vertices in $A$ are adjacent and hence the eigenvalues of $A$ are $|S|$ with multiplicity $1$ and $-1$ with multiplicity $|S|$. 

Now, we look into other components of $\Gamma(G)$. For this purpose, first we prove two claims that will be used later, which are as follows:\\ 
{\it Claim 1:} If $x \sim y$ in $\Gamma(G)$ and $\circ(x),\circ(y)\neq 2$, then $\circ(x)=\circ(y)$.\\
{\it Proof of Claim 1:} Suppose $\circ(x)=2^a$ and $\circ(y)=2^b$, where $a\neq b$. Without loss of generality, let $a<b$. Since $x\sim y$, we have $\circ(xy)=2$, i.e., $x^2=y^{-2}$, i.e., $e=x^{2^a}=y^{-2^a}$, which implies $\circ(y)\leq 2^a<2^b$, a contradiction. Hence the claim follows.\\
{\it Claim 2:} Let $B$ be a component in $\Gamma(G)$ and $e \not\in B$, i.e., $B\neq A$, then any two elements  in $B$ are of same order.\\
{\it Proof of Claim 2:} Let $x,y \in B$. Then there exists a path $x=x_0\sim x_1\sim x_2\sim \cdots \sim x_n=y$ joining $x$ and $y$. Thus, we have $\circ(x_0x_1)=\circ(x_1x_2)=\cdots=\circ(x_{n-1}x_n)=2$. So, multiplying them, we get $\circ((x_0x_n)(x^2_1x^2_2\cdots x^2_{n-1}))=2$. It is pertinent to recall that if the order of a product of two elements is $2$ in an abelian group, then either both of them has order $2$ or one of them has order $2$ and other is identity. If $\circ(xy)=\circ(x_0x_n)=1$, then $x$ and $y$, being mutually inverses, have same order. If $\circ(xy)=\circ(x_0x_n)=2$, then $x\sim y$. But as $x,y$ are not of order $2$, from Claim 1, $\circ(x)=\circ(y)$.

From Claim 2, we see that if two elements are in same component $B (\neq A)$, then they have the same order. However, elements of same order may belong to different components of $\Gamma(G)$. Let $B_1,B_2,\ldots,B_t$ be the components of $\Gamma(G)$ which consists of elements of order $4$.\\
{\it Claim 3:} If $x,y \in B_i$ with $x\neq y$ and $x\neq y^{-1}$, then $x\sim y$.\\
{\it Proof of Claim 3:} Since $x,y \in B_i$, then there exists a path $x=x_0 \sim x_1\sim x_2\sim \cdots \sim x_n=y$ joining $x$ and $y$, where $\circ(x_i)=4$ and $\circ(x_ix_{i+1})=2$ for $i=0,1,\ldots,n$. Also, we have $\circ(x_0x_1)=\circ(x_1x_2)=\cdots=\circ(x_{n-1}x_n)=2$, i.e., $\circ((x_0x_n)(x^2_1x^2_2\cdots x^2_{n-1}))=2$. Since, $\circ(x^2_i)=2$ and $x \neq y^{-1}$, we get $\circ(xy)=2$, i.e., $x\sim y$.\\
{\it Claim 4:} If $x,y \in B_i$, then $x\sim y$ if and only if $x\sim y^{-1}$.\\
{\it Proof of Claim 4:} It follows from the fact that  $x \sim y$ if and only if $\circ(xy)=2$, i.e., $(xy)^2=e$. i.e., $\circ(xy^{-1})=2$, i.e., $x \sim y^{-1}$.

From Claim 3 and 4, it follows that each $B_i$ contains even number of vertices, say $2m_i$ and if every element is paired with its inverse, then the adjacency matrix of $B_i$ takes the following form 
\begin{center}
         $  \mathcal{A}(B_i)=\begin{pmatrix}
               0 & 0 & \rvline  &  1 & 1 & \rvline & 1 & 1 & \rvline & \hdots & \rvline  & \hdots & \rvline  &  1 & 1 \\
               0 & 0 & \rvline & 1 & 1 & \rvline & 1 & 1 & \rvline & \hdots & \rvline  & \hdots & \rvline  & 1 & 1  \\
               \hline 1 & 1 & \rvline & 0 & 0 & \rvline & 1 & 1 & \rvline & \hdots & \rvline  & \hdots & \rvline   & 1 & 1 \\
               1 & 1 & \rvline & 0 & 0 & \rvline & 1 & 1 & \rvline &  \hdots & \rvline  & \hdots  & \rvline & 1 & 1\\
               \hline \vdots & \vdots & \rvline & \vdots &  \vdots & \rvline & \vdots & \vdots & \rvline & \ddots & \rvline  &   \ddots &  \rvline & \vdots & \vdots \\
               \hline \vdots & \vdots & \rvline & \vdots &  \vdots & \rvline & \vdots & \vdots & \rvline & \ddots & \rvline  &   \ddots & \rvline   & \vdots & \vdots \\
              
              \hline 1 & 1 & \rvline & 1 & 1 & \rvline & 1 & 1 & \rvline & \hdots & \rvline  & \hdots & \rvline & 0 & 0 \\
               1 & 1 & \rvline & 1 & 1 & \rvline & 1 & 1 & \rvline &  \hdots & \rvline  & \hdots & \rvline & 0 & 0 \\
           \end{pmatrix}.$
       \end{center}
From the matrix $\mathcal{A}(B_i)$, it is clear that $B_i$ is the complete $m_i$-partite graph with each partite-set consisting of two vertices, i.e.,  $B_i$ is the join $m_i$ copies of complements of $K_2$, i.e., $B_i=\overline{K_2} \vee \overline{K_2} \vee \cdots \vee \overline{K_2}$ ($m_i$ times). Hence, the eigenvalues of $B_i$ are $2m_i-2$ with multiplicity $1$, $0$ with multiplicity $m_i$ and $-2$ with multiplicity $m_i-1$, i.e., $\mathcal{A}(B_i)$ is integral.

Let $k\geq 3$ and $C_1,C_2,\ldots,C_l$ be the components of $\Gamma(G)$ which consists of elements of order $2^k$.\\
{\it Claim 5:} If $x,y \in C_i$ and $x\neq y$, then either $x\sim y$ or $x\sim y^{-1}$.\\
{\it Proof of Claim 5:} The proof follows in the same line as of the proof of Claim 3. \\
{\it Claim 6:} If $x\sim y_1$ and $x\sim y_2$ in $C_i$, then $y_1\nsim y_2$.\\
{\it Proof of Claim 6:} Since $\circ(xy_1)=\circ(xy_2)=2$, let $xy_1=z_1$ and $xy_2=z_2$ where $z_1$ and $z_2$ are elements of order $2$. Now, $xy_1xy_2=x^2y_1y_2=z_1z_2$ i.e. $y_1y_2=x^{-2}z_1z_2$ and $\circ(x^{-2}z_1z_2)\neq 2$. Hence, $\circ(y_1y_2)\neq 2$ i.e. $y_1\nsim y_2$.

Now, to get the adjacency matrix $\mathcal{A}(C_i)$ of size $2n_i\times 2n_i$ of $C_i$ in a suitable form, we choose $x \in C_i$ and index the rows of the matrix as $x,x^{-1},y^{-1}_1,y_1,y^{-1}_2,y_2,y^{-1}_3,y_3,\ldots$, where vertices are suitably renamed such that $x \sim y_i$ and $x \not\sim y^{-1}_i$, for all $i$. Then the matrix takes the following form:
\begin{center}
          $\mathcal{A}(C_i)=\begin{pmatrix}
              0 & 0 & \rvline & 0 & 1 & \rvline & 0 & 1 & \rvline & \hdots &  \rvline & \hdots  & \rvline & 0 & 1 \\
              
              0 & 0 & \rvline & 1 & 0 & \rvline & 1 & 0 & \rvline & \hdots & \rvline & \hdots  & \rvline & 1 & 0 \\
              \hline
              0 & 1 & \rvline & 0 & 0 & \rvline & 0 & 1 & \rvline & \hdots & \rvline & \hdots  & \rvline & 0 & 1 \\
              1 & 0 & \rvline & 0 & 0 & \rvline & 1 & 0 & \rvline & \hdots & \rvline & \hdots  & \rvline & 1 & 0 \\
              \hline
              \vdots & \vdots & \rvline & \vdots & \vdots & \rvline & \vdots & \vdots & \rvline & \ddots & \rvline & \ddots  & \rvline & \vdots & \vdots \\
               \hline
               \vdots & \vdots & \rvline & \vdots  & \vdots & \rvline & \vdots & \vdots & \rvline & \ddots & \rvline &\ddots  & \rvline & \vdots & \vdots \\
               \hline
              0 & 1 & \rvline & 0 & 1 & \rvline & 0 & 1 & \rvline & \hdots & \rvline & \hdots  & \rvline & 0 & 0 \\
             
              1 & 0 & \rvline & 1 & 0 & \rvline & 1 & 0 & \rvline & \hdots & \rvline & \hdots  & \rvline & 0 & 0 \\
              
          \end{pmatrix}.$
\end{center}
It can be easily observed that $\mathcal{A}(C_i)$ is the adjacency matrix of the graph  $\overline{K_n \square P_2}$. Now, as $K_n$ and $P_2$  are integral and $\overline{K_n \square P_2}$ is a regular graph, $\mathcal{A}(C_i)$ is also integral.

As all the components of $\Gamma(G)$ are integral, $\Gamma(G)$ is also integral.
 \end{proof}


We now introduce the notion of equitable partition of a graph and some related results, which will play an important role in proving our next result.

\begin{definition}(\cite{godsil-royle}, Section $9.3$)
Let $\Gamma$ be a graph of order $n$. A partition $\pi=\lbrace{C_1,C_2,\ldots,C_m}\rbrace$ of the vertex set $V(\Gamma)$ of $\Gamma$ is said to be equitable if for every pair of indices (not necessarily distinct) $i,j\in\lbrace{1,2,\ldots,m}\rbrace$, there is a non-negative integer $b_{i,j}$ such that each vertex $v$ in $C_i$ has exactly $b_{i,j}$ neighbours in $C_j$, irrespective of the choice of $v$.
\end{definition}

If $\pi=\lbrace{C_1,C_2,\ldots,C_m}\rbrace$ is an equitable partition of $\Gamma$, instead of $b_{i,j}$, we use the notation $[C_i,C_j]$ to denote the number of neighbours of each vertex of $C_i$ in $C_j$. For every pair $i,j$ this number is constant. Also, in general $[C_i,C_j]\ne[C_j,C_i]$.

\begin{definition}(\cite{godsil-royle}, Section $9.3$)
Let $\pi=\lbrace{C_1,C_2,\ldots,C_m}\rbrace$ be an equitable partition of a graph $\Gamma$. The matrix $\mathcal{Q}(\Gamma/\pi):=(b_{i,j})_{1\le i,j\le m}$ is said to be the quotient matrix or the partition matrix of $\Gamma$ relative to $\pi$.
\end{definition}

Unlike, the adjacency matrix, the quotient matrix of $\Gamma$ relative to an equitable partition may not be symmetric in general, since $b_{i,j}\ne b_{j,i}$.

\begin{theorem}(\cite{godsil-royle}, Theorem $9.3.3$)\label{char-poly-quotient-mat-div-graph}
Let $\pi$ be an equitable partition of a graph $\Gamma$. Then the characteristic polynomial of $\mathcal{Q}(\Gamma/\pi)$ divides the characteristic polynomial of $\mathcal{A}(\Gamma)$, where $\mathcal{A}(\Gamma)$ denotes the adjacency matrix of $\Gamma$.
\end{theorem}


\begin{lemma}\label{Zp-irr-eval-lemma}
The graph $\Gamma(\mathbb{Z}_p)$ has irrational eigenvalues, for $p\ne 2$.
\end{lemma}
\begin{proof}
The group $\mathbb{Z}_p$ has elements of order $1$ and $p$. We divide the elements of $\mathbb{Z}_p$ into two disjoint sets of elements with respect to their orders. Let $\pi =\{A_1,A_p\}$ be the partition of elements of order $1$ and $p$ in $\mathbb{Z}_p$, respectively. It is easy to verify that $\pi$ forms an equitable partition of $\Gamma(\mathbb{Z}_p)$. Clearly, $[A_1,A_1]=0$. Since the identity element is adjacent to each element of order $p$, we have $[A_1,A_p]=p-1$ and $[A_p,A_1]=1$. Finally, $[A_p,A_p]=p-3$, since an element of order $p$ is adjacent to every other element of order $p$ except for itself and its inverse. Hence, we form the quotient matrix with respect to the partition $\pi$ which is
$$
 \mathcal{Q}(\Gamma(\mathbb{Z}_{p})/\pi) =
 \begin{array}{cc}
 \begin{matrix} & A_1 & & A_p \end{matrix} \\
 \begin{matrix} A_1 \\ A_p \end{matrix} 
 \left(
 \begin{array}{cc|cc}
 0 & & p-1 \\
 \hline
 1 & & p-3
 \end{array}
 \right)
 \end{array}
 $$

The characteristic polynomial of this matrix is $x^2-x(p-3)-(p-1)=0$, whose roots are $\frac{(p-3)\pm\sqrt{p^2-2p+5}}{2}$, which are always irrational for any prime $p$. 
By Theorem \ref{char-poly-quotient-mat-div-graph}, it follows that $\Gamma(\mathbb{Z}_p)$ has irrational eigenvalues.
\end{proof}

\begin{theorem}
The graph $\Gamma(\mathbb{Z}_{p^n})$ has irrational eigenvalues, for $p\ne 2$ and $n\in\mathbb{N}$.
\end{theorem}
\begin{proof}
We show that the quotient matrix of $\Gamma(\mathbb{Z}_{p^n})$ with respect to the order partition of its elements has irrational eigenvalues. To prove this, we use induction on $n$. The base case is true by Lemma \ref{Zp-irr-eval-lemma}. Suppose the quotient matrix of $\Gamma(\mathbb{Z}_{p^{n-1}})$ with respect to the order partition $\pi=\lbrace{A_1,A_p,\ldots,A_{p^{n-1}}}\rbrace$ of the elements of $\mathbb{Z}_{p^{n-1}}$ has irrational eigenvalues. We claim that if $x,y\in\mathbb{Z}_{p^n}$ be such that $\circ(x)=p^i$ and $\circ(y)=p^j$ for some $0\le i,j\le n$ with $i\ne j$, then $\circ(xy)$ cannot be $p$. To establish this, we suppose $\circ(xy)=p$. Then $(xy)^p=1$, i.e., $x^p=(y^p)^{-1}$. Now, $\circ(x)=p^i$ implies $\circ(x^p)=p^{i-1}$ and $\circ(y)=p^j$ implies $\circ(y^p)=p^{j-1}$. Thus, $p^{i-1}=p^{j-1}$, which implies $i=j$, a contradiction. Hence, $[A_{p^i},A_{p^n}]=[A_{p^n},A_{p^i}]=0$, for $0\le i\le n-1$. Thus, the quotient matrix for $\Gamma(\mathbb{Z}_{p^n})$ with respect to the partition $\pi'=\pi\cup\lbrace{A_{p^n}}\rbrace$ becomes,
$$
 \mathcal{Q}(\Gamma(\mathbb{Z}_{p^n})/\pi') =
 \begin{array}{ccc}
  \begin{matrix} & & ~~~~A_1 & A_p & & \cdots & & A_{p^{n-1}}~~ A_{p^n} \end{matrix} \\
 \begin{matrix} A_1 \\ A_p \\ \\ \vdots \\ \\ A_{p^{n-1}} \\ \\ A_{p^n} \end{matrix} 
 \left(
 \begin{array}{ccccc|cc}
 & & & &\\
 & & & &\\
 & & & &\\
 & & \mathcal{Q}(\Gamma(\mathbb{Z}_{p^{n-1}})/\pi) & & & & \mathbf{0}\\
 & & & &\\
 & & & & \\
 & & & &\\
 \hline
 & & & &\\
 & & \mathbf{0} & & & \\
 \end{array}
 \right)
 \end{array}.
 $$

This shows that the eigenvalues of $\mathcal{Q}(\Gamma(\mathbb{Z}_{p^{n-1}})/\pi)$ are contained in the set of eigenvalues of $\mathcal{Q}(\Gamma(\mathbb{Z}_{p^n})/\pi')$. By induction hypothesis, it follows that $\mathcal{Q}(\Gamma(\mathbb{Z}_{p^n})/\pi')$ has irrational eigenvalues and eventually $\Gamma(\mathbb{Z}_{p^n})$ has irrational eigenvalues.
\end{proof}


\section{Laplacian Spectrum of $\Gamma(G)$}\label{lap-spectrum}
The Laplacian spectrum of a graph $\Gamma$ is the multiset of its Laplacian eigenvalues and is denoted by
$$\mathcal{L}\mbox{-}spec(\Gamma)=\begin{pmatrix}
\lambda_1 & \lambda_2 & \cdots & \lambda_k\\
r_1 & r_2 & \cdots &r_k
\end{pmatrix},$$
where the first row denotes the distinct Laplacian eigenvalues of $\Gamma$ and the second row denotes their respective multiplicities. In this section, we investigate the Laplacian spectrum of $\Gamma(G)$, where $G$ is an abelian group. To begin with, first we recall the notion of atoms and some basics of character theory of a finite group $G$. For more theory on representation and character theory of finite
groups, one may refer to \cite{martin-isaacs,steinberg}.

Let $G$ be a finite group. For $g\in G$, the \textit{atom} \cite{alperin-peterson} of $g$ is the set $[g]=\lbrace{x\in G:\langle{g}\rangle=\langle
{x}\rangle}\rbrace$ or equivalently, $[g]=\lbrace{g^k:\operatorname{gcd}(k,\circ(g))=1}\rbrace$. Clearly, for two elements $g,h\in G$, either $[g]=[h]$ or $[g]\cap[h]=\emptyset$. The set of all (non-equivalent) irreducible characters of $G$ is denoted by $\operatorname{Irr}(G)$. A character $\chi\in\operatorname{Irr}(G)$ is said to be \textit{real} if $\chi(g)\in\mathbb{R}$, for all $g\in G$. $\chi$ is said to be \textit{non-real} if $\chi(g)\notin\mathbb{R}$, for some $g\in G$. The character given by $\chi(g)=1$ for all $g\in G$ is said to be the \textit{trivial} character of $G$. For a set $S\subseteq G$ and for $\chi\in\operatorname{Irr}(G)$, let $\chi(S):=\sum_{s\in S}\chi(s)$. $S$ is called an $\textit{integral}$ set \cite{alperin-peterson} if $\chi(S)\in\mathbb{Z}$ for every $\chi\in\operatorname{Irr}(G)$. By convention, $\chi(\emptyset)=0$, for any $\chi\in\operatorname{Irr}(G)$. For any two groups $G_1$ and $G_2$, the irreducible characters of $G_1\times G_2$ are given by $\operatorname{Irr}(G_1\times G_2)=\lbrace{\chi\times\phi:\chi\in\operatorname{Irr}(G_1),\phi\in\operatorname{Irr}(G_2)}\rbrace$, where $(\chi\times\phi)(g_1,g_2)=\chi(g_1)\phi(g_2)$, for all $g_1\in G_1$ and $g_2\in G_2$.

We also recall the definition of Cayley sum graph of a group which will be used while computing the Laplacian spectrum of $\Gamma(G)$.

\begin{definition}
    For a finite abelian group $G$ and a subset $S$ of G, the Cayley sum graph $Cay^+(G, S)$
of $G$ with respect to $S$ is a graph with vertex set $G$ and two vertices $g$ and $h$ are
joined by an edge if and only if $g + h \in S$.
\end{definition}

We now state below two important results which will be useful in proving our theorem.

\begin{lemma}(\cite{alperin-peterson}, Proposition $4.1$)\label{chi(S)}
Let $G$ be a finite group and $g\in G$. Then for any $\chi\in\operatorname{Irr}(G)$, $\chi([g])\in\mathbb{Z}$. In other words, any atom of a finite group $G$ is an integral subset of $G$.
\end{lemma}

\begin{theorem}\label{abel-Cay-sum-spec-theorem}(\cite{nica}, Theorem 2.1)
Let $G$ be a finite abelian group and let $S\subseteq G$. Then the Laplacian eigenvalues of $\operatorname{Cay}^{+}(G,S)$, the Cayley sum graph of $G$ with respect to $S$, are given as follows:
\begin{enumerate}
\item[(i)] $|S|-\chi(S)$ for each real character $\chi$ of $G$,
\item[(ii)] $|S|\pm|\chi(S)|$ for each conjugate pair of non-real characters $\lbrace{\chi,\overline{\chi}}\rbrace$ of $G$.
\end{enumerate}
\end{theorem}

\begin{theorem}\label{abelian-laplacian-integral}
The prime order element graph $\Gamma(G)$ is Laplacian integral, for any finite abelian group $G$.
\end{theorem}

\begin{proof}
Let $S$ be the set of all prime order elements of $G$. Then it is easy to observe that $\Gamma(G)=\operatorname{Cay}^{+}(G,S)$. Now, if $x\in S$, then $\circ(x)=p$, for some prime $p$. Then it follows that $x,x^2,\ldots,x^{p-1}\in S$. In other words, $[x]\subseteq S$, which implies that $S$ is a union of some atoms of $G$. Since distinct atoms are disjoint, by Lemma \ref{chi(S)} we get $\chi(S)\in\mathbb{Z}$. Hence the result is a direct consequence of Theorem \ref{abel-Cay-sum-spec-theorem}.
\end{proof} 

In the subsequent results, we obtain the complete Laplacian spectrum of cyclic $p$-group for odd prime $p$ and describe the eigenvalues of $\mathbb{Z}_n$ for odd $n$ in general. We first recall the notion of Ramanujan sum. For integers $s\ge 0$ and $n\ge 1$, the \textit{Ramanujan sum} is defined by
$$c(s,n)=\sum_{(k,n)=1}\zeta_n^{sk},$$
where $\zeta_n=e^{\frac{2\pi i}{n}}$ denotes primitive $n$-th root of unity and $(k,n)$ denotes the gcd of $k$ and $n$. The value of $c(s,n)$ is known to be an integer (see \cite{ramanujan-sum} for reference).
$$c(s,n)=\frac{\varphi(n)}{\varphi\left(\frac{n}{(s,n)}\right)}\mu\left(\frac{n}{(k,n)}\right),$$
where $\varphi(\cdot)$ is the Euler's totient function and $\mu(\cdot)$ is the M{\"o}bius function. It is easy to see that $c(s,n)=\mu(n)$, if $(s,n)=1$ and $c(s,n)=\varphi(n)$, if $(s,n)=n$.

\begin{proposition}
Let $p$ be an odd prime and $r\in\mathbb{N}$. Then the Laplacian spectrum of $\Gamma(\mathbb{Z}_{p^r})$ is given by the following multiset
$$\mathcal{L}\mbox{-}spec(\Gamma(\mathbb{Z}_{p^r}))=\begin{pmatrix}
    0 & 2(p-1) & p & p-2\\
    \frac{1}{2}(p^{r-1}+1) & \frac{1}{2}(p^{r-1}-1) & \frac{1}{2}(p^r-p^{r-1}) & \frac{1}{2}(p^r-p^{r-1})
\end{pmatrix}.$$
\end{proposition}

\begin{proof}
Let $S=\lbrace{x\in \mathbb{Z}_{p^r}: \circ(x)=p}\rbrace$ be the set of all prime order elements in $\mathbb{Z}_{p^r}$. Then $|S|=p-1$. For $x\in S$, $x$ is of the form $x=tp^{r-1}$, where $t=1,\ldots,p-1$. The irreducible characters of $\mathbb{Z}_{p^r}$ are given by $\lbrace{\chi_k:k=0,\ldots,p^r-1}\rbrace$, where $\chi_k(a)=\zeta_{p^r}^{ak}$. Then for $0\le k\le p^r-1$,


$$\chi_k(S)=\sum_{x\in S}\chi_k(x)=\sum_{x\in S}\zeta_{p^r}^{kx}=\sum_{t=1}^{p-1}\zeta_{p^r}^{k(tp^{r-1})}=\sum_{t=1}^{p-1}\zeta_{p}^{kt}=c(k,p).$$

If $(k,p)=p$, then $c(k,p)=p-1$ and if $(k,p)=1$, then $c(k,p)=-1$. Since the only real character in $\mathbb{Z}_{p^r}$ is the trivial character $\chi_0$, it follows from Theorem \ref{abel-Cay-sum-spec-theorem} that $|S|-\chi_0(S)=(p-1)-(p-1)=0$, which is the eigenvalue of $\Gamma(\mathbb{Z}_{p^r})$ of multiplicity $1$ corresponding to the real character $\chi_0$.

Now $|\lbrace{k\in\mathbb{Z}_{p^r}:(k,p)=p}\rbrace|=p^{r-1}$. Then there are $(p^{r-1}-1)$ non-real characters $\chi_k$ of $\mathbb{Z}_{p^r}$ with $(k,p)=p$ and thus $\frac{1}{2}(p^{r-1}-1)$ conjugate pairs of non-real characters for the same. Again from Theorem \ref{abel-Cay-sum-spec-theorem}, for non-real characters $\chi_k$ with $(k,p)=p$,  
$|S|+\chi_k(S)=(p-1)+(p-1)=2(p-1)$ is an eigenvalue of $\Gamma(\mathbb{Z}_{p^r})$ with multiplicity $\frac{1}{2}(p^{r-1}-1)$ and $|S|-\chi_k(S)=(p-1)-(p-1)=0$ is an eigenvalue of $\Gamma(\mathbb{Z}_{p^r})$ with multiplicity $\frac{1}{2}(p^{r-1}-1)$.

Also, $|\lbrace{k\in\mathbb{Z}_{p^r}:(k,p)=1}\rbrace|=\varphi(p^r)$. Then there are $\frac{1}{2}(p^{r-1}-1)$ conjugate pairs of non-real characters $\chi_k$ with $(k,p)=1$. Again from Theorem \ref{abel-Cay-sum-spec-theorem}, for non-real characters $\chi_k$ with $(k,p)=1$,  
$|S|+\chi_k(S)=(p-1)+(-1)=p-2$ is an eigenvalue of $\Gamma(\mathbb{Z}_{p^r})$ with multiplicity $\frac{\varphi(p^r)}{2}$ and $|S|-\chi_k(S)=(p-1)-(-1)=p$ is an eigenvalue of $\Gamma(\mathbb{Z}_{p^r})$ with multiplicity $\frac{\varphi(p^r)}{2}$. This completes the proof.
\end{proof}
The next result describes the Laplacian eigenvalues of $\Gamma(\mathbb{Z}_n)$ for odd integer $n$. Before the statement, we introduce some notations for the sake of convenience. Let $k\in\mathbb{N}$, we denote by $[k]$ the set $\lbrace{1,2,\ldots,k}\rbrace$. For $1\le l\le k$, let $[k]\choose l$ denote the set of all $l$-subsets of $[k]$. The elements of $[k]\choose l$ are denoted as $A_{[k],l}^{(j)}$, for $j=1,2,\ldots,{k\choose l}$. Clearly, $A_{[k],k}=[k]$. Set $B_{[k],l}^{(j)}:=[k]\setminus A_{[k],l}^{(j)}$, for $j=1,2,\ldots,{k\choose l}$. Let $n$ be a positive integer with its canonical representation, that is $n=p_1^{r_1}p_2^{r_2}\cdots p_k^{r_k}$, where $p_i$'s are distinct primes and $r_i$'s are positive integers. For $T\subseteq [k]$, we define $p(T):=\sum_{i\in T}p_i$.


\begin{theorem}\label{Lap-spec-Zn-odd}
Let $n$ be an odd positive integer such that $n=p_1^{r_1}p_2^{r_2}\cdots p_k^{r_k}$ with $p_1<p_2<\cdots<p_k$ and $r_i\in\mathbb{N}$. Then the Laplacian eigenvalues of $\Gamma(\mathbb{Z}_n)$ are given by 
$$\left\lbrace{0, 2\left(p\left(A_{[k],k}\right)-k\right), p\left(A_{[k],k}\right)-p\left(A_{[k],l}^{(j)}\right), p\left(A_{[k],k}\right)+p\left(A_{[k],l}^{(j)}\right)-2k, p\left(A_{[k],k}\right)-2k, p\left(A_{[k],k}\right)}\right\rbrace,$$
where $l=1,2,\ldots,k-1$ and $j=1,2,\ldots,{k\choose l}$.
\end{theorem}

\begin{proof}
The set of all prime order elements of $\mathbb{Z}_n$ is given by 
$$S=\lbrace{(x_1,0,\ldots,0),(0,x_2,0,\ldots,0),\ldots,(0,\ldots,0,x_k):x_i\in\mathbb{Z}_{p_i^{r_i}}, \circ(x_i)=p_i,i=1,2,\ldots,k}\rbrace.$$
Clearly, $|S|=\sum_{i=1}^kp_i-k=p(A_{[k],k})-k$. Since $\operatorname{Irr}(\mathbb{Z}_n)=\lbrace{\chi_{(t_1,t_2,\ldots,t_k)}:0\le t_i\le p_i^{r_i}-1, i=1,2,\ldots,k}\rbrace$, where $\chi_{(t_1,t_2,\ldots,t_k)}=\chi_{t_1}\times\chi_{t_2}\times\cdots\times\chi_{t_k}$, and $\chi_{t_i}(0)=1$, for $i=1,2,\ldots,k$,
we have,

$$\chi_{(t_1,t_2,\ldots,t_k)}(S)=\sum_{\lbrace{x_1:\circ(x_1)=p_1}\rbrace}\chi_{t_1}(x_1)+\sum_{\lbrace{x_2:\circ(x_2)=p_2}\rbrace}\chi_{t_2}(x_2)+\cdots+\sum_{\lbrace{x_k:\circ(x_k)=p_k}\rbrace}\chi_{t_k}(x_k)$$
$$~~~~~~~=\sum_{\lbrace{x_1:\circ(x_1)=p_1}\rbrace}\zeta_{p_1^{r_1}}^{t_1x_1}+\sum_{\lbrace{x_2:\circ(x_2)=p_2}\rbrace}\zeta_{p_2^{r_2}}^{t_2x_2}+\cdots+\sum_{\lbrace{x_k:\circ(x_k)=p_k}\rbrace}\zeta_{p_k^{r_k}}^{t_kx_k}$$
$$=\sum_{l_1=1}^{p_1-1}\zeta_{p_1^{r_1}}^{t_1(l_1p_1^{r_1-1})}+\sum_{l_2=1}^{p_2-1}\zeta_{p_2^{r_2}}^{t_2(l_2p_2^{r_2-1})}+\cdots+\sum_{l_k=1}^{p_k-1}\zeta_{p_k^{r_k}}^{t_k(l_kp_k^{r_k-1})}$$
$$=\sum_{l_1=1}^{p_1-1}\zeta_{p_1}^{t_1l_1}+\sum_{l_2=1}^{p_2-1}\zeta_{p_2}^{t_2l_2}+\cdots+\sum_{l_k=1}^{p_k-1}\zeta_{p_k}^{t_kl_k}~~~~~~~~~~~~~~~~~~~~~~$$
$$=c(t_1,p_1)+c(t_2,p_2)+\cdots+c(t_k,p_k)~~~~~~~~~~~~~~~~~~~~~~~~$$

If $(t_i,p_i)=p_i$, then $c(t_i,p_i)=p_i-1$ and if $(t_i,p_i)=1$, then $c(t_i,p_i)=-1$. We now introduce a new notion called gcd tuple. Define a relation $\rho$ on $\mathbb{Z}_n\cong\mathbb{Z}_{p_1}^{r_1}\times\mathbb{Z}_{p_2}^{r_2}\times\cdots\times\mathbb{Z}_{p_k}^{r_k}$ by $(t_1,t_2,\ldots,t_k)\rho(s_1,s_2,\ldots,s_k)$ if and only if $gcd(t_i,p_i)=gcd(s_i,p_i)$ for all $i=1,2,\ldots,k$. Note that the gcd's are either $1$ or $p_i$. Clearly, $\rho$ is an equivalence relation on $\mathbb{Z}_n$. We call these equivalence classes \textit{gcd tuples} and they can be represented as $$[(\varepsilon_1,\varepsilon_2,\ldots,\varepsilon_k)]=\lbrace{(t_1,t_2,\ldots,t_k)\in \mathbb{Z}_{p_1}^{r_1}\times\mathbb{Z}_{p_2}^{r_2}\times\cdots\times\mathbb{Z}_{p_k}^{r_k}:gcd(t_i,p_i)=\varepsilon_i,i=1,2,\ldots,k}\rbrace,$$ where $\varepsilon_i=1$ or $p_i$. We associate each tuple $(t_1,t_2,\ldots,t_k)$ from the gcd tuple $[(\varepsilon_1,\varepsilon_2,\ldots,\varepsilon_k)]$ to the character $\chi_{(t_1,t_2,\ldots,t_k)}$ of $\mathbb{Z}_n$. Now we apply Theorem \ref{abel-Cay-sum-spec-theorem}:

\textit{1. Eigenvalues from real character:}

The only real character of $\mathbb{Z}_n$, where $n$ is odd is the trivial character $\chi_{(0,\ldots,0)}$ which hails from the gcd tuple $[(p_1,p_2,\ldots,p_k)]$. Hence,
$$|S|-\chi_{(0,\ldots,0)}(S)=\left(\sum_{i=1}^kp_i-k\right)-\left(\sum_{i=1}^kp_i-k\right)=0.$$

\textit{2. Eigenvalues from non-real characters:}\\

\textit{Case $1$:} Gcd tuple $[(p_1,p_2,\ldots,p_k)].$

\begin{enumerate}
\item [(i).] $$|S|-\chi_{(t_1,t_2,\ldots,t_k)}(S)=\left(\sum_{i=1}^kp_i-k\right)-\left(\sum_{i=1}^kp_i-k\right)=0.$$

\item [(ii).] $$|S|+\chi_{(t_1,t_2,\ldots,t_k)}(S)=\left(\sum_{i=1}^kp_i-k\right)+\left(\sum_{i=1}^kp_i-k\right)=2\left(\sum_{i=1}^kp_i-k\right)=2\left(p\left(A_{[k],k}\right)-k\right).$$
\end{enumerate}

\textit{Case $2$:} Gcd tuple $[(1,\ldots,1,p_{i_1},1,\ldots,1,p_{i_2},1,\ldots,1,p_{i_k},1,\ldots,1)]$, where $1\le l<k.$\\

Let $A_{[k],l}^{(j)}\in{[k]\choose l}$ be an arbitrary $l$-subset of $[k]$, for $j=1,2,\ldots,{k\choose l}$. For this subset, we have
\begin{enumerate}
\item [(i).] $$|S|-\chi_{(t_1,t_2,\ldots,t_k)}(S)=\left(p\left(A_{[k],k}\right)-k\right)-\left(p\left(A_{[k],l}^{(j)}\right)-l-(k-l)\right)=p\left(A_{[k],k}\right)-p\left(A_{[k],l}^{(j)}\right).$$

\item [(ii).] $$|S|+\chi_{(t_1,t_2,\ldots,t_k)}(S)=\left(p\left(A_{[k],k}\right)-k\right)+\left(p\left(A_{[k],l}^{(j)}\right)-l-(k-l)\right)=p\left(A_{[k],k}\right)+p\left(A_{[k],l}^{(j)}\right)-2k.$$

\end{enumerate}

\textit{Case $3$:} Gcd tuple $[(1,\ldots,1)].$\\

\begin{enumerate}
\item [(i).] $$|S|-\chi_{(t_1,t_2,\ldots,t_k)}(S)=\left(\sum_{i=1}^kp_i-k\right)-(-k)=p\left(A_{[k],k}\right).$$

\item [(ii).] $$|S|+\chi_{(t_1,t_2,\ldots,t_k)}(S)=\left(\sum_{i=1}^kp_i-k\right)+(-k)=p\left(A_{[k],k}\right)-2k.$$
\end{enumerate}
This completes the proof.
\end{proof}

\begin{example}
We illustrate Theorem \ref{Lap-spec-Zn-odd} by taking the group $\mathbb{Z}_n$, where $n=315=3^2\cdot 5\cdot 7$, i.e., $p_1=3,p_2=5,p_3=7$. Then 
$${[3]\choose 1}=\lbrace{\lbrace{1}\rbrace,\lbrace{2}\rbrace,\lbrace{3}\rbrace}\rbrace, {[3]\choose 2}=\lbrace{\lbrace{1,2}\rbrace,\lbrace{2,3}\rbrace,\lbrace{1,3}\rbrace}\rbrace.$$
Thus we have, 
$$p\left(A_{[3],3}\right)=p_1+p_2+p_3=15,$$
$$p\left(A_{[3],1}^{(1)}\right)=p_1=3,p\left(A_{[3],1}^{(2)}\right)=p_2=5,p\left(A_{[3],1}^{(3)}\right)=p_3=7,$$
$$p\left(A_{[3],2}^{(1)}\right)=p_1+p_2=8,p\left(A_{[3],2}^{(2)}\right)=p_2+p_3=12,p\left(A_{[3],2}^{(3)}\right)=p_1+p_3=10.$$
Hence, the eigenvalues of $\mathbb{Z}_{315}$ are $\lbrace{0,3,5,7,8,9,10,12,14,15,16,17,19,21,24}\rbrace$.
\end{example}

\begin{remark}
The Laplacian eigenvalue $2\left(p\left(A_{[k],k}\right)-k\right)$ of $\Gamma(\mathbb{Z}_n)$ described in Theorem \ref{Lap-spec-Zn-odd} does not occur in the spectrum if $n$ is square-free. For example, if we take $n=105=3\cdot 5\cdot 7$, then the eigenvalue $2\left(p\left(A_{[3],3}\right)-3\right)=24$ does not occur in the spectrum of $\Gamma(\mathbb{Z}_{105})$. In this case, the set of eigenvalues of $\Gamma(\mathbb{Z}_{105})$ becomes $\lbrace{0,3,5,7,8,9,10,12,14,15,16,17,19,21}\rbrace$.
\end{remark}

\section{Planarity of $\Gamma(G)$}\label{planar-section}
In this section, we characterize the groups $G$ for which $\Gamma(G)$ is planar. The key result which is used for this is the characterization of planar graphs in terms of forbidden graphs, now known as Kuratowski's theorem: {\it A finite graph is planar if and only if it does not contain a subgraph that is a subdivision of the complete graph $K_5$ or the complete bipartite graph $K_{3,3}$.} 
\begin{theorem}
Let $G$ be a finite group such that $\Gamma(G)$ is planar. Then $G$ is one of the following:
\begin{itemize}
    \item $G\cong \mathbb{Z}_5$.
    \item $G$ is a cyclic $3$-group.
    \item $G$ is a $2$-group with with exactly one or three elements of order $2$.
    \item $|G|=2^m3^n$ with a unique subgroup each of order $2$ and $3$.
    
\end{itemize}
\end{theorem}

\begin{proof}
    Let $G$ be a finite group such that  $\Gamma(G)$ is planar. We first show that $|G|$ has no prime factor $p\geq 7$. Suppose $p\geq 7$ divides $|G|$. Then $G$ has a subgroup $H=\langle a \rangle \cong \mathbb{Z}_p$. Then $H$ has at least $6$ distinct non-identity elements $X=\{a,a^{-1},b,b^{-1},c,c^{-1}\}$ of order $p$. Consider the subgraph induced by $X\cup\{e\}$ given in Figure \ref{fig:1} (left). The vertices $a^{-1},b^{-1},c^{-1}$ can be identified to get a minor isomorphic to $K_5$ and hence $\Gamma(H)$ and thereby $\Gamma(G)$ is non-planar.
    Thus, if $p$ is a prime dividing $|G|$, then $p\leq 5$. 
\begin{figure}[htb]
\centering
\begin{tikzpicture}[scale=0.5]
	
\draw[fill= black, draw = black] (0,0) circle (0.1);
\node[] at (0.2,-0.4) {\bf $e$};
\node[] at (0,1.6) {\bf $a$};
\node[] at (0,-4.6) { $a^{-1}$};
\node[] at (1.4,-0.8) {\bf $b$};
\node[] at (-1.4,-0.8) {\bf $c$};
\node[] at (-3.7,2.6) {\bf $b^{-1}$};
\node[] at (3.9,2.6) {\bf $c^{-1}$};

\foreach \a in {90,210,330} { 
\draw[fill] (\a:1.2cm) circle (0.1);
\draw[fill] (\a:1.2 cm) -- (0,0);
\draw[fill] (\a:1.2 cm) -- (\a + 120:1.2 cm);}

\foreach \a in {30,150,270} { 
\draw[fill] (\a:4.2cm) circle (0.1);
\draw[fill] (\a:4.2 cm) -- (0,0);
\draw[fill] (\a:4.2 cm) -- (\a + 120:4.2 cm);}

\draw[fill] (30:4.2 cm) -- (90:1.2 cm);
\draw[fill] (150:4.2 cm) -- (90:1.2 cm);
\draw[fill] (30:4.2 cm) -- (330:1.2 cm);
\draw[fill] (270:4.2 cm) -- (330:1.2 cm);
\draw[fill] (270:4.2 cm) -- (210:1.2 cm);
\draw[fill] (150:4.2 cm) -- (210:1.2 cm);

\node[] at (0,-7) { Subgraph induced by $X\cup\{e\}$};

\begin{scope}[xshift=14cm,yshift=-1cm]
\draw[fill= black, draw = black] (0,0) circle (0.1);
\foreach \a in {45,135,225,315} { 
\draw[fill] (\a:3cm) circle (0.1);
\draw[fill] (\a:3 cm) -- (0,0);
\draw[fill] (\a:3 cm) -- (\a + 90:3 cm);}

\foreach \a in {75,105,255,285} { 
\draw[fill] (\a:4cm) circle (0.1);
\draw[fill] (\a:4 cm) -- (0,0);
}
\draw[fill] (75:4 cm) -- (105:4 cm);
\draw[fill] (255:4 cm) -- (105:4 cm);
\draw[fill] (75:4 cm) -- (285:4 cm);
\draw[fill] (255:4 cm) -- (285:4 cm);

\draw[fill] (75:4 cm) -- (60:4 cm);
\draw[fill] (60:4 cm) -- (45:3 cm);
\draw[fill] (105:4 cm) -- (120:4 cm);
\draw[fill] (120:4 cm) -- (135:3 cm);
\draw[fill] (255:4 cm) -- (240:4 cm);
\draw[fill] (225:3 cm) -- (240:4 cm);
\draw[fill] (300:4 cm) -- (285:4 cm);
\draw[fill] (300:4 cm) -- (315:3 cm);

\foreach \a in {60,120,240,300} { 
\draw[fill=red,draw=red] (\a:4cm) circle (0.1);
}

\node[] at (5:0.5cm) { $e$};
\node[] at (45:3.5cm) { $h$};
\node[] at (135:3.5cm) { $h^3$};
\node[] at (225:3.5cm) { $h^4$};
\node[] at (315:3.5cm) { $h^2$};

\node[] at (75:4.5cm) { $k$};
\node[] at (105:4.5cm) { $k^3$};
\node[] at (255:4.5cm) { $k^4$};
\node[] at (285:4.5cm) { $k^2$};

\node[] at (305:4.7cm) { $h^3k^3$};
\node[] at (235:4.6cm) { $hk$};
\node[] at (125:4.7cm) { $h^2k^2$};
\node[] at (55:4.7cm) { $h^4k^4$};
\node[] at (4,-6) { Subgraph induced by $H\cup K \cup \{e\}\cup \{h^ik^i:1\leq i \leq 4\}$};
\node[] at (4.5,-0.5) { Minor};
\draw[fill,->] (3,0) -- (6,0);
\end{scope}

\begin{scope}[xshift=23cm,yshift=-1cm]
\draw[fill= black, draw = black] (0,0) circle (0.1);
\foreach \a in {45,135,225,315} { 
\draw[fill] (\a:3cm) circle (0.1);
\draw[fill] (\a:3 cm) -- (0,0);
\draw[fill] (\a:3 cm) -- (\a + 90:3 cm);}

\foreach \a in {75,105,255,285} { 
\draw[fill] (\a:4cm) circle (0.1);
\draw[fill] (\a:4 cm) -- (0,0);
}
\draw[fill] (75:4 cm) -- (105:4 cm);
\draw[fill] (255:4 cm) -- (105:4 cm);
\draw[fill] (75:4 cm) -- (285:4 cm);
\draw[fill] (255:4 cm) -- (285:4 cm);

\draw[fill] (75:4 cm) -- (45:3 cm);
\draw[fill] (105:4 cm) -- (135:3 cm);
\draw[fill] (255:4 cm) -- (225:3 cm);
\draw[fill] (315:3 cm) -- (285:4 cm);


\end{scope}

\end{tikzpicture}
\caption{Non-planar subgraphs of $\Gamma(G)$}
\label{fig:1}
\end{figure}
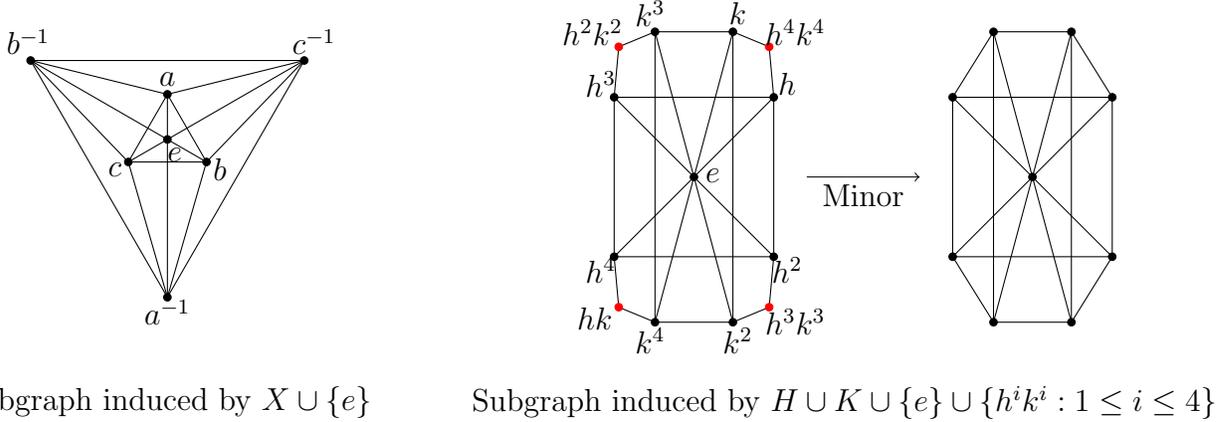

Suppose $5$ divides $|G|$. If $5^2$ divides $|G|$, then $G$ has a subgroup isomorphic to $\mathbb{Z}_{25}$ or $\mathbb{Z}_5\times \mathbb{Z}_5$. However, as $\Gamma(\mathbb{Z}_{25})$ and $\Gamma(\mathbb{Z}_5\times \mathbb{Z}_5)$ are not planar, $5^2$ does not divide $|G|$. Let $H=\langle h \rangle$ and $K=\langle k \rangle$ (if they exist) be two distinct subgroups of order $5$ in $G$. Then $|HK|=25$. Consider the subgraph of $\Gamma(G)$ obtained by some of the vertices in $HK$ (Figure \ref{fig:1} (right)). Merging the \textcolor{red}{red} vertices with one of their neighbors, we get the above minor of $\Gamma(G)$ which is not planar. Thus $G$ has a unique normal subgroup of order $5$. If $|G|$ has any other prime factor apart from $5$, it must be $2$ or $3$. That is, $G$ has a subgroup of order $10$ or $15$. However, as none of $\Gamma(D_5),\Gamma(\mathbb{Z}_{10}),\Gamma(\mathbb{Z}_{15})$ are non-planar, it follows that $|G|=5$, i.e., $G\cong \mathbb{Z}_5$.

If $5$ does not divide $|G|$, then $|G|=2^n$ or $3^n$ or $2^m3^n$. If $|G|=3^n$, we show that $G$ is cyclic. For this, we recall a result: {\it Let $p$ be an odd prime and $\mathcal{G}$ be a non-cyclic group of order $p^k$ $(k\geq 3)$, then $\mathcal{G}$ has a non-cyclic subgroup of order $p^{k-1}$.} If $G$ is non-cyclic, then by repeated use of the above result, we get a non-cyclic subgroup $\mathbb{Z}_3 \times \mathbb{Z}_3$ in $G$. Now, as $\Gamma(\mathbb{Z}_3 \times \mathbb{Z}_3)$ is not planar, we get a contradiction. In return we obtain that, $G$ is cyclic.

Next we deal with the case when $|G|=2^n$. Let $n_2$ denote the number of elements of order $2$ in $G$. Let $n_2\geq 3$ and $\circ(x)=\circ(y)=2$ with $\circ(xy)=m$. If $m>2$, then $\langle x,y\rangle\cong D_m\leq G$, i.e., $m=2^k$, where $k\geq 2$. Then $G$ has a subgroup isomorphic to $D_4$ and as $\Gamma(D_4)$ is not planar, we get a contradiction. Hence $\circ(xy)=2$. Now, if $n_2\geq 5$, we get at least $5$ elements of order $2$ such that the order of product of any two of them is also $2$. Hence, these five elements of order $2$ form a complete graph on $5$ vertices ($K_5$) in $\Gamma(G)$, a contradiction. So $n_2\leq 3$, i.e., $n_2=1$ or $3$. It is to be noted that $n_2=1$ implies that $G\cong \mathbb{Z}_{2^n}$ or $Q_{2^n}$. On the other hand, finite $2$-groups with exactly three involutions has been classified in \S 82 in \cite{berkovich}.

Finally, let $|G|=2^m3^n$. Suppose $G$ has more than one element of order $2$, then by arguing as above, if $\circ(x)=\circ(y)=2$, then $\circ(xy)=2$ and hence $G$ has a normal subgroup $\{e,x,y,xy\}\cong K_4$. Also, $G$ has a subgroup isomorphic $\mathbb{Z}_3$ of order $3$. Then $G$ contains a subgroup isomorphic to $K_4\rtimes \mathbb{Z}_3$ or $K_4\times \mathbb{Z}_3$. But, as $\Gamma(K_4\rtimes \mathbb{Z}_3)$ and $\Gamma(K_4\times \mathbb{Z}_3)$ are non-planar, $G$ must have a unique element of order $2$, say $x$. Let $H=\langle x \rangle$. As $G$ can not contain any subgroup isomorphic to $S_3$, any element of order $3$ in $G$ commutes with $x$. Next we claim that $G$ has a unique subgroup of order $3$. If not, let $K=\langle k\rangle$ and $L=\langle l\rangle$ be two subgroups of order $3$ in $G$. As elements of order $2$ and $3$ commutes in $G$, we have $HK\cong HL\cong \mathbb{Z}_6$ but $HK\neq HL$ and $HK\cap HL=H$. Thus, $|(HK)(HL)|=|HKL|=\frac{6\cdot 6}{2}=18$ and $$HKL=\{e,h,k,k^2,l,l^2,hk,hl,hk^2,hl^2,hkl,hk^2l,hkl^2,hk^2l^2,kl,k^2l,kl^2,k^2l^2\}.$$
It can be checked that subgraph induced by $HKL$ is not planar, hence $G$ is not planar, a contradiction. Thus, $G$ has unique subgroup of order $3$. 
\end{proof}
\begin{theorem}
     Let $|G|=2^n$. If $G$ has no subgroup isomorphic to $D_4$ and ${\mathbb{Z}^3_2}$, then $\Gamma(G)$ is planar.
\end{theorem}
\begin{proof}
    We first show that $G$ has either $1$ or $3$ elements of order $2$. If not, let $\{x_1,x_2,\ldots,x_t\}$ be the set of elements of order $2$ in $G$ where $t\geq 5$. If for any distinct $i,j$, $\circ(x_ix_j)=2^l>2$, then $\langle x_i,x_j\rangle \cong D_{2^l}$ has a subgroup isomorphic to $D_4$. Hence,  $\circ(x_ix_j)=2$ for all $i\neq j$. Thus, elements of order $2$ along with identity forms an elementary abelian $2$-group of order $\geq 8$, i.e., $G$ has a subgroup isomorphic to ${\mathbb{Z}^3_2}$, a contradiction. Consequently, $G$ has $1$ or $3$ elements of order $2$.

    If $G$ has exactly one element of order $2$, then $G\cong \mathbb{Z}_{2^n}$ or $Q_{2^n}$, and clearly $\Gamma(G)$ is planar. So, we assume that $G$ has exactly three elements of order $2$, say $z_1,z_2,z_3$. It is to be noted that, in this case, identity along with three elements of order $2$ forms a component $C$ of $\Gamma(G)$. 
    
    {\it Claim 1:} If $x,y$ are two vertices of $G\setminus C$ such that $x\sim y$, then $\circ(x)=\circ(y)$. 

    {\it Proof of Claim 1:} If possible, let $\circ(x)=2^a$ and $\circ(y)=2^b$, where $a>b$. We first show that $xy=yx$. As $x\sim y$, we have $\circ(xy)=\circ(yx)=2$. Also, $\circ(x^{2^{a-1}})=2$. Then $xy,yx,x^{2^{a-1}}$ are elements of order $2$. Clearly, $x^{2^{a-1}}\neq xy$ or $yx$ as that would imply $\circ(x)=\circ(y)$. Moreover, if $xy\neq yx$, then $xy,yx,x^{2^{a-1}}$ are three distinct elements of order $2$ in $G$ and hence their product $(yx)(x^{2^{a-1}})(xy)=e$, i.e., $x^{2^{a-1}+2}=y^{-2}$. Squaring both sides, we get $x^4=y^{-4}$, i.e., $\circ(x)=\circ(y)$, a contradiction. Thus only possibility left is $xy=yx$.

    Therefore, we have $e=(xy)^{2^b}=x^{2^b}\cdot y^{2^b}=x^{2^b}\neq e$ as $b<a$, a contradiction. Eventually, the claim follows.

    From Claim 1, it is obvious that any two elements in the same component of $\Gamma(G)$, other than $C$, must be of the same order.

    {\it Claim 2:} Any component containing an element of order $4$ is a $4$-cycle.
    
    {\it Proof of Claim 2:} Let $x$ be an element of order $4$. Then $\circ(x^2)=2$, i.e., $x^2 \in \{z_1,z_2,z_3\}$ and $deg(x)=3-1=2$ in $\Gamma(G)$. Let $x^2=z_1$. Then we get a $4$-cycle $x\sim x^{-1}z_2\sim x^{-1}\sim x^{-1}z_3\sim x$. Now, as every element in this component is of order $4$, all of them have degree $2$ and hence the component itself is the above $4$-cycle. Thus the claim holds.

    {\it Claim 3:} Any component containing an element of order $2^t$ with $t\geq 3$ is isomorphic to a cube, i.e., $C_4\square P_2$.
    
    {\it Proof of Claim 3:} Let $x$ be an element of order $2^t$ with $t\geq 3$. Then degree of every vertex in the component of $x$ is $3$ and $x^{2^{t-1}}$ is an element of order $2$. Without loss of generality, let $z_1=x^{2^{t-1}}$. Then $xz_1=z_1x$. Let $H=\{e,z_1,z_2,z_3\}$, which implies that $H$ is a normal subgroup of $G$ and thus $xH=Hx$, i.e., $\{x,xz_1,xz_2,xz_3\}=\{x,z_1x,z_2x,z_3x\}$, i.e., $\{xz_2,xz_3\}=\{z_2x,z_3x\}$. 
    
    Suppose $xz_2=z_3x$ and $xz_3=z_2x$. Then the component takes the following form:
    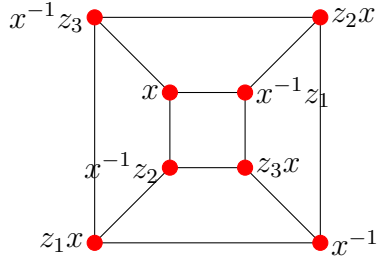
\begin{figure}[htb]
        \centering
        \begin{tikzpicture}
   \draw [] (0, 0) -- (0, -1);
   \draw [] (0, 0) -- (1, 0);
   \draw [] (1, 0) -- (1, -1);
   \draw [] (0, -1) -- (1, -1);
   \draw [] (0, 0) -- (-1, 1);
   \draw [] (1, 0) -- (2, 1);
   \draw [] (0, -1) -- (-1, -2);
   \draw [] (1, -1) -- (2, -2);
   \draw [] (-1,1) -- (-1,-2);
   \draw [] (-1, 1) -- (2,1);
   \draw [] (2, 1) -- (2,-2);
   \draw [] (-1, -2) -- (2,-2);
    \node [left] at (0, 0) {$x$};
			\draw [fill, red] (0, 0) circle[radius = 1mm];
   \node [right] at (1, 0) {$x^{-1}z_1$};
			\draw [fill, red] (1, 0) circle[radius = 1mm];
   \node [left] at (0, -1) {$x^{-1}z_2$};
			\draw [fill, red] (0, -1) circle[radius = 1mm];
   \node [right] at (1, -1) {$z_3x$};
			\draw [fill, red] (1, -1) circle[radius = 1mm];
   \node [left] at (-1, 1) {$x^{-1}z_3$};
			\draw [fill, red] (-1, 1) circle[radius = 1mm];
   \node [right] at (2, 1) {$z_2x$};
			\draw [fill, red] (2, 1) circle[radius = 1mm];
   \node [right] at (2, -2) {$x^{-1}$};
			\draw [fill, red] (2, -2) circle[radius = 1mm];
   \node [left] at (-1, -2) {$z_1x$};
			\draw [fill, red] (-1, -2) circle[radius = 1mm];
   
   \end{tikzpicture} 
        \caption{Component isomorphic to a cube, i.e.,  $C_4\square P_2$}
        \label{fig:2}
    \end{figure}

 Clearly, we cannot add any more vertex in this component as it will contradict the fact that each vertex is of degree $3$. Thus the claim holds.

 From the above two claims, it follows that $\Gamma(G)$ is the disjoint union of a $K_4$ (complete graph on $4$ vertices containing $e$) and some copies of $C_4$ (consisting of elements of order $4$) and some copies of $C_4\square P_2$. Hence $\Gamma(G)$ is planar.
 \end{proof}
 
 \section{Cliques in $\Gamma(G)$}\label{clique-section}
In this section, we compute the clique number of $\Gamma(G)$, when $G$ is an abelian $p$-group.
 \begin{theorem}
     Let $G$ be an abelian $2$-group and $H$ be a maximal elementary abelian subgroup of $G$. If $H\cong {\mathbb{Z}^t_2}$, then $\omega(\Gamma(G))=\omega(\Gamma(H))=2^t$. 
 \end{theorem}
 \begin{proof}
    Clearly $H$ is a clique in $\Gamma(G)$ of size $2^t$. If possible, let $M=\{x_1,x_2,\ldots ,x_l\}$ be a maximum clique in $\Gamma(G)$ with $l>2^t$. Thus, $\circ(x_ix_j)=2$ for all $i,j\in \{1,2,\ldots,l\}$ with $i\neq j$. 

    As $\circ(x_1x_2)=\circ(x_2x_3)=2$ and $G$ is abelian, their product is of order $1$ or $2$. In the former case, we have $x_1x_3{x}^2_2=e$ i.e., ${x}^{-2}_2=x_1x_3$. Also, $(x_1x_2)^2=e$ i.e., ${x}^2_1={x}^{-2}_2$. Thus, ${x}^2_1={x}^{-2}_2=x_1x_3$ i.e., $x_1=x_3$, a contradiction. Therefore, $\circ(x_1{x}^2_2x_3)=2$ i.e., $(x_1x_3)^2{x}^4_2=e$ i.e., ${x}^4_2=e$. So $\circ(x_2)=1 \mbox{ or } 2 \mbox{ or } 4$. 

    If $\circ(x_2)=4$, and as ${x}^2_2{x}^2_j=e$ for all $j=1,2,\ldots,l$, this implies ${x}^2_j={x}^{-2}_2={x}^2_2$ i.e., $${x}^2_i={x}^2_j, \forall i,j \mbox{ and } \circ(x_i)=4, \forall i.$$ 

    As $G$ is abelian, we have $T=\langle x_1,x_2,\ldots,x_l\rangle\cong \mathbb{Z}_4 \times \mathbb{Z}^{l-1}_2$. Note that $T$, and hence $G$, has a subgroup isomorphic to $\mathbb{Z}^{l}_2$. Thus $\Gamma(G)$ has a clique of size $2^l$. Also by the given condition, $l\leq t$. Then $2^l \leq 2^t<l$, a contradiction. So $\circ(x_2)\neq 4$. Note that in similar way, it can be shown that $\circ(x_i)\neq 4$ for all $i=1,2,\ldots,l$. Hence  $\circ(x_i)=1$ or $2$, for all $i$. Then $T=\langle x_1,x_2,\ldots,x_l\rangle\cong \mathbb{Z}^{l-1}_2$. Again, by the given condition, we have $l-1\leq t$, i.e., $2^{l-1}\leq 2^t<l$, a contradiction.
    
    Subsequently, $\Gamma(G)$ has no clique of size larger than $2^t$ and the result follows.    
 \end{proof}


  \begin{theorem}
     Let $G$ be an abelian $p$-group ($p$ being an odd prime) and $H$ be a maximal elementary abelian subgroup of $G$. If $H\cong {\mathbb{Z}^k_p}$, then $\omega(\Gamma(G))=\omega(\Gamma(H))=\dfrac{p^k+1}{2}$. 
 \end{theorem}
 \begin{proof}
    Clearly $\Gamma(G)$ has a clique of size $\dfrac{p^k+1}{2}$. If possible, let $M=\{x_1,x_2,\ldots,x_l\}$ be a maximal clique in $\Gamma(G)$ with $l>\dfrac{p^k+1}{2}$. As $M$ is a clique, we have $\circ(x_ix_j)=p$ for all $i,j~(i\neq j)$. As $\circ(x_1x_2)=\circ(x_2x_3)=p$, we have $\circ(x_1x_3x_2^2)=1 \mbox{ or } p$. 

    If $\circ(x_1x_3x_2^2)=1$, then $x_2^{-2}=x_1x_3$ i.e., $x_2^{-2p}=1$. Thus, $\circ(x_2)$ divides $2p$ i.e., $\circ(x_2)=1 \mbox{ or } p$. If $\circ(x_2)=1$, then $x_1x_3=e$, which is a contradiction and hence $\circ(x_2)=p$. Now, any $x_i$ can be written as $x_i=(x_ix_2)x_2^{-1}$ (product of two elements of order $p$) which implies $\circ(x_i)=p \mbox{ or } 1$ for all $i$. Hence, $T=\langle x_1,x_2,\ldots,x_l \rangle\cong \mathbb{Z}^{l-1}_p$ is an elementary abelian subgroup of $G$ and by the given condition, $l-1\leq k$. This contradicts $l>\dfrac{p^k+1}{2}$. So, we have $\circ(x_1x_3x_2^2)=p$, then $(x_1x_3)^px_2^{2p}=e$ i.e., $x_2^{2p}=e$ i.e., $\circ(x_2)=1 \mbox{ or } p$. By similar arguments, we get a contradiction. Hence $\Gamma(G)$ has no clique of size greater than $\dfrac{p^k+1}{2}$.
 \end{proof}

 \section{Conclusion and Open Issues}\label{conclusion-section}
 In this paper, we proved a few results on planarity, and adjacency and Laplacian spectrum of prime order element graph of a group. Although the result on planarity is a complete classification, the results on adjacency and Laplacian spectrum are specified to specific families, namely cyclic and abelian groups respectively. We conclude with some open issues and speculations, which were numerically observed using GAP/SAGE \cite{sage} computations.

 \begin{enumerate}
     \item Theorem \ref{abelian-2-group-integral} proves that abelian $2$-groups yield integral adjacency spectrum. A partial converse of this also seems to be true, i.e., {\it if $p$ be an odd prime dividing $|G|$, then $\Gamma(G)$ has an irrational adjacency eigenvalue.}
     \item Theorem \ref{abelian-laplacian-integral} proves that for any abelian group $G$, $\Gamma(G)$ is Laplacian integral. However, we strongly believe that this is true for any group, in general, i.e., {\it $\Gamma(G)$ is Laplacian integral for any group $G$.}
 \end{enumerate}

\section*{Statements and Declarations}
The second author is supported by the funding of UGC [NTA Ref. No. 211610129182], Govt. of India.  The third author acknowledges the funding of DST-FIST Sanction no. $SR/FST/MS-I/2019/41$ and DST-SERB-MATRICS Sanction no. $MTR/2022/000020$, Govt. of India. 

\subsection*{Data Availability Statements}
Data sharing not applicable to this article as no datasets were generated or analysed during the current study.

\subsection*{Competing Interests} The authors have no competing interests to declare that are relevant to the content of this article.

\end{document}